\documentclass[a4paper,12pt]{article}
\usepackage{amssymb}
\usepackage{amsmath}
\usepackage{amsthm}
\usepackage{xcolor}
\usepackage{hyperref}
\hypersetup{
	colorlinks,
	linkcolor={red!60!black},
	citecolor={green!60!black},
	urlcolor={blue!60!black}
}

\usepackage{geometry}
\newtheorem{Theorem} {Theorem} [section]
\newtheorem{Proposition} [Theorem] {Proposition}
\newtheorem{Lemma} [Theorem] {Lemma}
\newtheorem{Remark} [Theorem] {Remark}
\newtheorem{Definition} [Theorem] {Definition}

\newcommand{\Ff}{{\mathbb F}}
\newcommand{\cC}{{\mathcal C}}
\newcommand{\cS}{{\mathcal S}}
\newcommand{\PG}{\mathrm{PG}}
\newcommand{\rk}{\mathrm{rk}}
\newcommand{\Tr}{\mathrm{Tr}}
\newcommand{\<}{\langle}
\renewcommand{\>}{\rangle} % was: tabbing command

\title{An improved algebraic construction for Ramsey numbers}
\author{\renewcommand\thefootnote{\alph{footnote}}
Ferdinand Ihringer\footnotemark[1] \and
\renewcommand\thefootnote{\alph{footnote}}
Sam Mattheus\footnotemark[2]}

\date{22 Aug 2026}

\begin{document}
\maketitle

{\renewcommand\thefootnote{\alph{footnote}}
\footnotetext[1]{Dept.~of Mathematics,
Southern University of Science and Technology, Shenzhen, Guangdong, China.
E-mail: {\tt ihringer@sustech.edu.cn}}

\footnotetext[2]{Dept.~of Mathematics and Data Science, 
	Vrije Universiteit Brussel, 
	Pleinlaan 2, 1050 Brussels, Belgium.
	E-mail: {\tt Sam.Mattheus@vub.be}}}

\begin{abstract}
  We provide an explicit algebraic construction showing that, uniformly for integers $3 \leq s \leq t$, as $t \to \infty$,
  \begin{align*}
   R( s,t ) \geq t^{(1-o(1)) \log s / \log(\log s  + 1) }.
  \end{align*}
  For large fixed $s$, this improves the dependence on $s$ in the general
  off-diagonal construction of Alon and Pudl\'ak.
  In particular, $R(33, t) \geq t^{2.1-o(1)}$, to our knowledge,
  the first explicit construction showing $R(s, t) \geq t^c$ for some fixed $s$
  and some $c > 2$.
  In the diagonal case,
  it improves the leading constant in the exponent of the classical
  Frankl--Wilson bound from $1/4$ to $1$, while being almost as simple
  to describe.
\end{abstract}

% MSC2020: 05D10, 05B25
% Keywords: Keywords: explicit Ramsey graphs, off-diagonal Ramsey numbers, point-hyperplane incidence, p-rank

\section{Introduction}

Ramsey theory is an area of mathematics underpinned by the philosophy that in any large
enough structure, there exists a relatively large uniform substructure. Originating in Ramsey's work in mathematical logic \cite{Ramsey29}, the subject has developed into a central area of combinatorics, with connections to number theory, geometry, topology, theoretical computer science, and ergodic theory. Classical manifestations include Schur's theorem on monochromatic solutions to $x+y=z$ \cite{Schur17}, van der Waerden's theorem on monochromatic arithmetic progressions \cite{vdW27}, and Rado's theorem on partition-regular systems of linear equations \cite{Rado33}. For surveys of recent developments in graph Ramsey theory, we refer to Morris \cite{Morris26} and Verstra\"ete \cite{Verstraete25}.

For positive integers $s,t$, the \textit{Ramsey number} $R(s, t)$
is the smallest integer $R$ such that every graph on $R$ vertices
contains either a clique of size $s$ or an independent set of size $t$.
Equivalently, $R(s,t)>N$ if and only if there exists a graph on $N$ vertices with clique number less than $s$ and independence number less than $t$.
In 1935, Erd\H{o}s and Szekeres \cite{ES35} famously established the bound
\begin{align}
  R(s, t) \leq \binom{s+t-2}{t-1}. \label{eq:trivup}
\end{align}

The strongest asymptotic lower bounds for Ramsey numbers have generally relied on the probabilistic method. In his pioneering 1947 paper \cite{Erdos47}, Erd\H{o}s used this method to find the lower bound $R(s,s)\geq 2^{(1/2-o(1))s}$ for \textit{diagonal} Ramsey numbers, where $s=t\to\infty$.
Although subsequent work improved lower-order factors, the leading exponential rate in this bound has remained unchanged for almost eighty years.
By contrast, recent breakthroughs have substantially improved the upper bound \eqref{eq:trivup}.
Campos, Griffiths, Morris, and Sahasrabudhe \cite{CGMS26} obtained the first exponential improvement over the Erd\H{o}s--Szekeres bound, and Gupta, Ndiaye, Norin, and Wei \cite{GNNW24} subsequently sharpened their estimate. 
Most recently, Balister et al.\ \cite{BBCGHMST26} gave a shorter proof of these results that also extends to the multicolour setting.

\medskip

No explicit construction is known that proves $R(s,s)\geq C^s$ for any absolute constant $C>1$.
Erd\H{o}s repeatedly highlighted this gap, offering a \$100 prize for such a construction, see \cite{ChungGraham98}.
The classical benchmark is the modular-intersection construction of
Frankl and Wilson.

\begin{Theorem}[Frankl--Wilson \cite{FW81}] \label{thm:FW}
	There exist direct explicit constructions of graphs showing that, as $s \to \infty$,
	\[
	R(s,s)\geq
	s^{(\frac14-o(1))\frac{\log s}{\log\log s}}.
	\]
\end{Theorem}

Throughout, $\log$ denotes the \textit{logarithmus dualis}, i.e., logarithm base $2$.
Alon \cite{Alon98} and Grolmusz \cite{Grolmusz00} later provided explicit direct constructions of the same logarithmic order of magnitude.

In this paper, a \textit{direct construction} refers to an algebraically or geometrically defined family whose vertex set and adjacency relation admit a short description.
This is in contrast to a different source of explicit constructions emerging in theoretical computer science from two-source dispersers and extractors. In this setting, a family of graphs on $N$ vertices is called \textit{strongly explicit} if, given the labels of two vertices $u$ and $v$, one can determine in time $\operatorname{polylog}(N)$ whether $u$ and $v$ are adjacent.
This line of research includes the breakthrough result by Barak et al.~\cite{BRSW12} and culminated in the recent result of Li \cite{Li23}, who showed that there exists $\varepsilon > 0$ and a family of \textit{strongly explicit} graphs showing $R(s,s)\geq 2^{s^\varepsilon}$ for sufficiently large $s$.
This is substantially stronger than the bounds arising from the classical algebraic constructions, although it still falls short of Erd\H{o}s' challenge of obtaining an exponential lower bound.
Our focus, however, is on direct algebraic and geometric constructions
in the preceding sense, rather than those obtained from the machinery of extractors.

\medskip

A second classical regime is that of \textit{off-diagonal} Ramsey numbers, where $s \geq 3$ is fixed and $t \to \infty$.
Here, the upper bound \eqref{eq:trivup} specializes to the asymptotic bound $R(s,t) = O(t^{s-1})$.
Recently, Brada\v{c} \cite{Bradac26} essentially resolved this regime by finding a construction based on finite geometry which is tight up to some polylogarithmic factors.
If we again limit ourselves to explicit constructions, then the picture is much different.
In 1994, Alon gave an explicit construction for $R(3, t) = \Omega(t^{1.5})$ \cite{Alon94}, which is still the best known.
Later, Kostochka, Pudl\'ak, and R\"odl gave explicit constructions for
$R(4, t) = \Omega(t^{1.6})$, $R(5, t) = \Omega(t^{1.\overline{6}})$, and
$R(6, t) = \Omega(t^2)$ \cite{KPR10}.
The strongest general explicit construction is due to Alon and Pudl\'ak.

\begin{Theorem}[Alon--Pudl\'ak \cite{AP01}] \label{thm:AP}
	For every fixed $s \geq 3$, there exists an explicit construction of graphs
	that shows, as $t \to \infty$,
	\[
	R(s, t) \geq t^{\Omega(\sqrt{\log s / \log \log s})}.
	\]
\end{Theorem}

We now state our main result, simultaneously improving the classical constructions of Frankl--Wilson and Alon--Pudl\'ak.

\begin{Theorem}\label{thm:main}
	Let $t \geq s \geq 3$ and $t \rightarrow \infty$.
	Then there exists a direct explicit construction of graphs
	that shows
	\begin{align*}
		R( s,t ) \geq t^{(1-o(1)) \lceil \log s \rceil / \log(\lceil \log s \rceil + 1) }.
	\end{align*}
\end{Theorem}

Here $o(1)\to 0$ as $t \to \infty$, uniformly for $3 \leq s \leq t$.
This improves the leading constant in the exponent in Theorem \ref{thm:FW} and squares the exponent in Theorem \ref{thm:AP}.
Moreover, our construction provides the explicit lower bound $R(33, t) \geq t^{2.1-o(1)}$, the smallest case that improves on $R(s,t) \geq R(6, t) = \Omega(t^2)$ for $s \geq 6$.
To the best of our knowledge, this is the first explicit construction that shows
$R(s, t) \geq t^c$ for some fixed $s$ and some $c > 2$.
It would be interesting to find better explicit constructions for small $s$ that surpass the $t^2$ barrier, which appears to be a natural obstacle for this problem.

\medskip

The idea behind our result is that in the case of explicit constructions, controlling the independence number is the natural obstacle. 
We immediately deal with this issue by constructing a graph on point-hyperplane flags in projective space, where we know \textit{a priori} that a $p$-rank argument
gives a good bound on the independence number, see Proposition \ref{prop:ind}. 
Similar ideas also lie at the heart of the earlier constructions by Frankl--Wilson, Alon, and Grolmusz, but their vertex sets are set-systems, whereas ours have a more finite-geometric flavour.

In this way, our construction is perhaps closer in spirit to that of Brada\v{c}, although we use a different adjacency relation, and to recent work of Bamberg et al.\ \cite{BBIR26} who investigated a similar approach.
Choosing the right set of point-hyperplane flags as the vertex set can be seen as a deterministic analogue of the idea of random sampling in finite-geometric settings, ubiquitous in recent results in off-diagonal Ramsey theory, see \cite{Bradac26,MV24,Verstraete25}.

\paragraph*{AI Declaration}

While the authors had the idea to use $p$-rank arguments as the starting point, the first construction (see Remark \ref{remark:firstconstruction}) was provided by ChatGPT 5.6, along with a bound on its clique number based on algebraic geometry.
The authors subsequently developed, refined, and independently verified
all arguments and take full responsibility for the contents of the paper.

\section{The construction}

Let $d \geq 3$.
Write $q=2^h$ for some positive integer $h$ and 
write $\Tr$ for the trace from $\Ff_{q^d}$ to $\Ff_q$.
Put $V = \Ff_{q^d}$. We will also
identify $V$ with the $d$-dimensional vector space over $\Ff_q$.
In particular, the $1$-dimensional subspaces $\< x \>$, or $1$-spaces for short, of $V$ are the points of the
projective geometry $\PG(d-1, q)$. In the following, \textit{point}
will always refer to a $1$-space $\< x \>$ of $V$.
In general, we will treat vector spaces projectively throughout this work.

Choose $\beta \in V$ such that
\[
\beta, \beta^q, \ldots, \beta^{q^{d-1}}
\]
is a basis of $V$, that is, $\beta$ is a normal element over $\Ff_q$.
Put $a = \beta+\beta^q$.

\begin{Definition}
	The graph $TG^\prec_{d,h}$ has as vertex set the points of $\PG(d-1,q)$ with a given total ordering $\prec$,
	where $\<x\>$ and $\<y\>$ with $\<x\> \prec \< y \>$
	are adjacent if $\Tr(ax/y) = 0$.
\end{Definition}

In the next few sections we will show that $\omega(TG^\prec_{d,h}) \leq 2^{d-2}$ (Theorem \ref{thm:clq}) and $\alpha(TG^\prec_{d,h}) \leq d^h+1$ (Proposition \ref{prop:ind}). Given these results, our main theorem now easily follows.

\begin{proof}[Proof of Theorem \ref{thm:main}]
	Put $m:=\lceil\log s\rceil$ and $d:=m+1$ and, for all sufficiently large $t$, let $
	h:=\left\lfloor\frac{\log(t-2)}{\log d}\right\rfloor$.
	Then $\omega(TG^\prec_{d,h}) \leq 2^{d-2} =2^{m-1} <s$ and $\alpha(TG^\prec_{d,h}) \leq d^h+1 \leq t-1 <t$.
	Thus, $TG^\prec_{d,h}$ contains neither a clique of size $s$
	nor an independent set of size $t$.
	
	The number of vertices of $TG^\prec_{d,h}$ is
	\[
	\frac{2^{hd}-1}{2^h-1}
	\geq 2^{h(d-1)}
	=2^{hm}.
	\]
	Since $s \leq t$, we have $\log d = o(\log t)$ and hence
	\[
	h
	\geq \frac{\log(t-2)}{\log d}-1
	=(1-o(1))\frac{\log t}{\log d}
	\]
	as $t \to \infty$. It follows that
	\[
	R(s,t)
	\geq 2^{hm} \geq 
	t^{(1-o(1))\lceil\log s\rceil/
		\log(\lceil\log s\rceil+1)}. \qedhere
	\]
\end{proof}

\begin{Remark}\label{remark:firstconstruction}
	The first construction suggested by ChatGPT 5.6, which we will denote by $TG_{d,h}$, is the following.
	The vertex set consists of the points of $\PG(d-1,q)$, $\< x \>$ and $\< y\>$ adjacent if $\Tr(ax/y) = \Tr(ay/x) = 0$.
	Essentially the same proofs show $\omega(TG_{d,h}) \leq \binom{d-2}{\lceil (d-2)/2 \rceil}$ and $\alpha(TG_{d,h}) \leq (d^h+1)^2$, thus leading to a slightly worse leading factor in the exponent of the corresponding Ramsey bound.

	This graph is a Cayley graph over the cyclic group of order $\frac{q^d-1}{q-1}$ that corresponds to a Singer cycle in $\PG(d-1,q)$.
	Thus, eigenvalues can be estimated with standard methods. From small parameter computations and general heuristics, it appears that $TG_{d,h}$ is optimally pseudorandom when either $d$ is fixed and $h \rightarrow \infty$,
	or when $d \rightarrow \infty$ and $h$ is fixed.
\end{Remark}

\begin{Remark}
	Before proving our results, we give some remarks and a posteriori observations about our construction.
	\begin{enumerate}
		\item The same argument works in general characteristic (with $a = \beta-\beta^q$), but gives quantitatively worse lower bounds as the $p$-rank bound in Proposition \ref{prop:ind} becomes weaker.
		Moreover, our results are independent of the choice of order $\prec$.
		
		\item For a point $\< x \>$ of $V$, define
		\[
		H_x = \{ \< y \>: \Tr(xy) = 0 \}.
		\]
		Note that this parameterizes the hyperplanes of $V$. 
		One can then immediately observe, using the fact that $\Tr(a)=0$, that $(\<x\>, H_{a/x})$ is an incident point-hyperplane pair. From this point of view, the vertex set of $TG_{d,h}$ comprises the edges in a perfect matching $M$ of the point-hyperplane incidence graph $G$ in $\PG(d-1,q)$. 
		It then follows that $TG^\prec_{d,h}$ is the induced subgraph $H_G^\prec[M]$, borrowing the notation of Kostochka, Pudl\'ak and R\"odl \cite{KPR10}.
	\end{enumerate}
\end{Remark}

\section{Independence number}

Let $B$ be the point-hyperplane incidence matrix of $\PG(d-1, q)$
with rows and columns ordered according to $\prec$ so
that the diagonal entries correspond to $(\< x\>, H_{a/x})$.

\begin{Proposition}\label{prop:ind}
 We have $\alpha(TG^\prec_{d,h}) \leq d^h+1$.
\end{Proposition}
\begin{proof}
  Write $\rk_2(M)$ for the $2$-rank of a matrix $M$.
  It is well-known that $B$ has $\rk_2(B) = d^h+1$, see  \cite{Smith69}.
  If $\<x\>$ and $\<y\>$ are nonadjacent in $TG^\prec_{d,h}$,
  then the principal submatrix of $B$ induced by $\<x\>$ and $\<y\>$ is one of
  \[
   \begin{pmatrix}
    1 & 0 \\
    0 & 1
   \end{pmatrix},
   \qquad
   \begin{pmatrix}
    1 & 0 \\
    1 & 1
   \end{pmatrix}.
  \]
  Thus, if $C$ is the principal submatrix of $B$ induced by
  an independent set $Y$ of $TG^\prec_{d,h}$, then
  \[
    |Y| = \rk_2(C) \leq \rk_2(B) = d^h+1. \qedhere
  \]
\end{proof}

\section{Clique number}

For the remainder of this section, write $\overline{V}$ for the $d$-dimensional
vector space over the algebraic closure of $\Ff_{q^d}$.
For $R$ a subspace of $\overline{V}$, we denote by $R^\perp$ its orthogonal
complement with respect to the standard inner product $\cdot$.
Write $r = \dim(R)$ and suppose that $1 \leq r \leq d-1$.
Denote the Hadamard product by $\circ$.
Call a vector $\lambda = (\lambda_0, \ldots, \lambda_{d-1}) \in \overline{V}$
a \textit{minimal zero-sum sequence} if $\sum_i \lambda_i = 0$ and for all nonempty proper subsets $S$ of $\{ 0, \ldots, d-1 \}$,
we have
\[
 \sum_{i \in S} \lambda_i \neq 0.
\]

We refer to Chapter I of Hartshorne's book \cite{Hartshorne}
for all the required algebraic geometry.
The key points are that a nonconstant rational function on a
complete nonsingular curve has a zero and a pole (see \cite[Exercise I.6.4]{Hartshorne}), and that it admits a Laurent expansion in a local parameter (essentially, \cite[Theorem I.5.5A]{Hartshorne}).

\begin{Lemma}\label{lem:pairs}
 Let $\lambda \in \overline{V}$ be a minimal zero-sum sequence. Then the set
 \[
  Z_\lambda := \{ (\<x\>, \< y\>): \< x \circ y \> = \< \lambda \> \},
 \]
  where $\<x\>$ is a point of $R$
 and $\<y\>$ is a point of $R^\perp$, is finite.
\end{Lemma}

\begin{proof}
Suppose for a contradiction that $Z_\lambda$ contains an irreducible curve $\cC$.
We will work over the function field of $\cC$ and rescale $x$ and $y$ so that
\[
 x_i y_i = \lambda_i
\]
for all $i$. As $\lambda$ is a minimal zero-sum sequence, all $\lambda_i$ are nonzero.

The point $\<x\>$ cannot be constant on $\cC$, as otherwise $\<y\>$ would be too.
Hence, some $x_i/x_j$ is not constant.
Thus, $x_i/x_j$ either has a zero or a pole at some place $P$.
Near $P$, choose a local parameter $u$ and write
\[
 x_i = u^{v_i} (c_i + O(u))
\]
for integers $v_i$ and constants $c_i$.
As $x_i/x_j$ has a zero or pole at $P$, the integers $v_i$ are not all equal.
(Otherwise, there is no zero or pole at $u=0$.)
After multiplying $x$ by an appropriate power of $u$ and $y$ by the inverse of that power of $u$,
we may assume that
\[
 \min_i v_i = 0 \qquad \text{ and } \qquad M := \max_i v_i > 0.
\]
 Put $S := \{ i: v_i = M \}$.
 Note that $S$ is a nonempty proper subset of $\{ 0, \ldots, d-1 \}$.
 Write
 \[
  x = \sum_j X_j u^j, \qquad y = \sum_j Y_j u^j
 \]
 with coefficients $X_j, Y_j \in \overline{V}$.
 Since $x \in R$, we have $X_j \in R$. Similarly, $Y_j \in R^\perp$.
 Hence, $X_M \cdot Y_{-M} = 0$.
 For $i \in S$, write $x_i = c_i u^M + \cdots$ with $c_i \neq 0$.
 Similarly, write $y_i = \lambda_i c_i^{-1} u^{-M} + \cdots$.
 For $i \notin S$, the coefficient of $u^{-M}$ in $y_i$ is zero. Consequently,
 \[
  \sum_{i \in S} \lambda_i = X_M \cdot Y_{-M} = 0.
 \]
 But $S$ is nonempty and proper, so this contradicts that $\lambda$ is a minimal zero-sum sequence.
\end{proof}

Now we can quantify what \textit{finite} in Lemma \ref{lem:pairs} means
by counting the common isolated zeros on a Segre variety.
It is essentially a variant of B\'ezout's theorem.

\begin{Proposition}\label{prop:fincurve}
  Suppose that $\lambda$ is a minimal zero-sum sequence.
  Then the number of pairs $(\<x\>, \<y\>)$ with $\<x\>$ is a point of $R$
 and $\<y\>$ is a point of $R^\perp$ satisfying $\<x \circ y\> = \<\lambda\>$
 is at most $\binom{d-2}{r-1}$.
\end{Proposition}

\begin{proof}
Let $\cS$ be the image of $R \times R^\perp$ (seen projectively)
under the Segre embedding.
The diagonal Segre coordinates are $z_i := x_iy_i$.
Since $x \in R$ and $y \in R^\perp$,
\[
 \sum_{i=0}^{d-1} z_i = \sum_{i=0}^{d-1} x_i y_i = x \cdot y = 0.
\]
Recall that being a minimal zero-sum sequence implies that every $\lambda_i$ is nonzero.
Thus, for $i = 1, \ldots, d-1$, we can define the bilinear forms
\begin{align}
 F_i(x, y) := \lambda_0 x_i y_i - \lambda_i x_0 y_0 = \lambda_0 z_i - \lambda_i z_0. \label{eq:linF}
\end{align}
Then
\begin{align*}
 \sum_{i=1}^{d-1} F_i &= \lambda_0 \sum_{i=1}^{d-1} z_i - z_0 \sum_{i=1}^{d-1} \lambda_i
 = -\lambda_0 z_0 + \lambda_0 z_0 = 0.
\end{align*}
Hence, $F_1 = \cdots = F_{d-2} = 0$ implies $F_{d-1} = 0$.
Thus, the common zero set of $F_1, \ldots, F_{d-2}$ is
\[
 Z_\lambda \cup B,
\]
where $Z_\lambda$ is defined as in Lemma \ref{lem:pairs}
and $B = \{ (\< x \>, \< y \>): x \circ y = 0 \}$.
Here $Z_\lambda$ and $B$ are disjoint:
If $z_0 \neq 0$, then
\[
 z_i = \frac{z_0}{\lambda_0} \lambda_i
\]
for all $i$, so $\< x \circ y\> = \< \lambda \>$.
If $z_0 = 0$, then $z_i = 0$ for all $i$, so $x \circ y = 0$.
By Lemma \ref{lem:pairs}, $Z_\lambda$ is finite.

Before we conclude the proof, let us recall some standard
algebraic geometry.
If $X$ is an $m$-dimensional projective variety, then,
by the iterative application of a variant of B\'ezout's Theorem,
see \cite[Theorem I.7.7]{Hartshorne},
$m$ hyperplanes have at most $\deg X$ isolated points on $X$.

Each bilinear form $F_1, \ldots, F_{d-2}$
defines a hyperplane section of $\cS$, cf.\ Equation \eqref{eq:linF}.
Furthermore, $X = \cS$, $m = d-2$ (as $R$ has projective dimension $r-1$ and $R^\perp$ has projective dimension $d-r-1$), and $\deg X = \binom{d-2}{r-1}$ (see \cite[Exercise I.7.1(b)]{Hartshorne}).
As $Z_\lambda \subseteq \{ x_0y_0 \neq 0 \}$ and $B \subseteq \{ x_0y_0 = 0 \}$,
every point in $Z_\lambda$ is an isolated common zero of $F_1, \ldots, F_{d-2}$.
Hence,
\[
 |Z_\lambda| \leq \binom{d-2}{r-1}. \qedhere
\]
\end{proof}

For $i \in \{ 0, \ldots, d-1\}$,
define
\[
 \lambda_i = a^{q^i} = \beta^{q^i} + \beta^{q^{i+1}}.
\]
The special choice of $a$ as $a = \beta + \beta^q$ with $\beta$ normal
is only needed for the following.

\begin{Lemma}\label{lem:sums}
 The vector $\lambda = (\lambda_0, \ldots, \lambda_{d-1})$ is a minimal zero-sum sequence.
\end{Lemma}
\begin{proof}
  For any subset $S$ of $\{ 0, \ldots, d-1 \}$,
  \[
   \sum_{i \in S} \lambda_i = \sum_{j=0}^{d-1} (1_S(j) + 1_S(j-1)) \beta^{q^j},
  \]
  where $1_S(-1) = 1_S(d-1)$. As the conjugates of $\beta$ form a basis,
  the sum vanishes precisely if $1_S(j) = 1_S(j-1)$ for all $j$,
  that is, if $S$ is either empty or the whole set.
\end{proof}

Now we can prove our bound on the clique number.

\begin{Theorem}\label{thm:clq}
 For every $d \geq 3$ and $h \geq 1$ we have $\omega(TG^\prec_{d,h}) \leq 2^{d-2}$.
\end{Theorem}

\begin{proof}
 Let $C = \{ \< x_1 \>, \ldots, \< x_m \> \}$ be a clique of $TG^\prec_{d,h}$,
 where $\<x_i\> \prec \<x_j\>$ if $i < j$.
 Let $R_j = \< x_i: i \in \{ 1, \ldots, j \} \>$ and $r_j := \dim_{\mathbb{F}_q}(R_j)$. 
 By the definition of $TG^\prec_{d,h}$, we obtain $\langle a/x_j \rangle \in R_j^\perp$ for all $j \in [m]$, where orthogonality is defined by $\Tr$.
 Since $\langle a/x_j \rangle$ is nonzero and belongs to $R_j^\perp$, we have
 $r_j\leq d-1$ for every $j\in[m]$.
 
 Let $R$ be a space appearing in the chain $R_1 \subseteq \dots \subseteq R_m$, $r := \dim_{\mathbb{F}_q}(R)$ its $\Ff_q$-dimension, and $J :=\{j\in[m]:r_j=r\}$ the collection of indices so that $R_j = R$ for all $j \in J$.
 
 Then all vertices $\langle x_j \rangle$ with $j \in J$ satisfy
 \[
  \langle x_j \rangle \in R \qquad \text{ and } \qquad \langle a/x_j \rangle \in R^\perp.
 \]
 Define a map $\theta:\Ff_{q^d} \to \Ff_{q^d}^d$ by
 \[
 \theta(z) = (z, z^q, \ldots, z^{q^{d-1}}) \qquad \text{ and } \qquad \overline{R} = \< \theta(v): \langle v \rangle \in R \>_{\overline{\Ff}_{q^d}} \subseteq \overline{V}.
 \]
 Note that $\Tr(zw) = \theta(z) \cdot \theta(w)$ is the standard dot product on $\Ff_{q^d}^d$.
 For all $\langle v \rangle \in R$ and $j \in J$ we have
 \[
 \theta(v) \cdot \theta(a/x_j) = \Tr(av/x_j) = 0.
 \]
 Since the vectors $\theta(v)$ span $\overline{R}$,
 we have $\theta(a/x_j) \in \overline{R}^\perp$.
 Moreover, $\theta(x_j) \circ \theta(a/x_j) = \theta(a) = \lambda$.
 Thus, every $\<x_j\> \in C$ with $j \in J$ supplies a pair $(\< \theta(x_j)\>, \<\theta(a/x_j)\>)$ counted by Proposition \ref{prop:fincurve}.
 It is easily verified that these pairs are distinct.
 Thus, there are at most $\binom{d-2}{r-1}$ such points $\langle x_j \rangle$.
 Thus,
 \[
  |C| \leq \sum_{r=1}^{d-1} \binom{d-2}{r-1} = 2^{d-2}. \qedhere
 \]
\end{proof}

\section{Conclusion}

While our construction improves on previous algebraic constructions, the gap to the upper bounds \eqref{eq:trivup} remains significant.
In the diagonal regime, it remains to be seen if these kinds of constructions are able to match the extractor-based lower bounds yielding $R(s,s) \geq 2^{s^\varepsilon}$, or provide an affirmative answer to the aforementioned question of Erd\H{o}s.

In the off-diagonal regime, Alon and Pudl\'ak asked \cite{AP01} whether there is an explicit
construction that shows $R(s, t) \geq t^{\Omega(s)}$ for fixed $s$.
Given the lack of progress in roughly a quarter century,
already the following would be very interesting: is there $\varepsilon > 0$ such that there exists an explicit construction yielding $R(s, t) \geq t^{\Omega(s^{\varepsilon})}$?

\paragraph*{Acknowledgements}
This research was supported by National Key R\&D Program of China under grant number 2025YFA1017700.
The second author was supported by postdoctoral fellowship 1267923N from the Research Foundation Flanders (FWO),
and is very grateful to SUSTech for their hospitality during the visit in which this work was carried out.
We would like to acknowledge helpful comments from Noga Alon and Qing Xiang.

\end{document}